\documentclass{lmcs} 
\pdfoutput=1

\usepackage{lastpage}
\lmcsdoi{22}{3}{1}
\lmcsheading{}{\pageref{LastPage}}{}{}%
{Aug.~21,~2025}{Jul.~14,~2026}{}

\usepackage{hyperref}
\usepackage[utf8]{inputenc}
\def\fish{\mbox{\rm Fish}}
\def\head{\mbox{\rm Head}}
\def\tail{\mbox{\rm Tail}}

\begin{document}
\title[Compactness and non-measurable sets]{Constraint satisfaction problems, compactness and non-measurable sets}
\author{Claude Tardif\lmcsorcid{0000-0003-3855-989X}}
\address{Department of Mathematics and Computer Science,
Royal Military College of Canada,
Kingston, Ontario, Canada.}	
\email{Claude.Tardif@rmc.ca}  

\begin{abstract}
A finite relational structure $A$ is called {\em compact} if 
for any infinite
relational structure $B$ of the same type, the existence 
of a homomorphism
from $B$ to $A$ is equivalent to the existence of homomorphisms 
from all finite
substructures of $B$ to $A$. We show that if $A$ has width $1$, 
then the compactness
of $A$ can be proved in the axiom system of Zermelo and Fraenkel,
but otherwise,
the compactness of $A$ implies the existence of non-measurable sets 
in $3$-space.
\end{abstract}

\maketitle
\section{Introduction} \label{intro}

A graph is bipartite if and only if it does not contain an odd cycle.
Already in 1916, K\"{o}nig~\cite{konig} acknowledged that this theorem requires the axiom
of choice for the infinite case. One of the two possible colourings must be selected for each
connected component. Later in 1961, Mycielski~\cite{mycielski} proved that in the infinite case, K\"{o}nig's theorem 
is precisely equivalent to the axiom of choice for sets of pairs. Now, consider the statement that a digraph admits a homomorphism
to the transitive tournament $T_2$ on two vertices if and only if it admits no homomorphism from the
directed path $P_2$ with two consecutive arcs. Even in the infinite case, it requires no version of the
axiom of choice. Indeed, if there are no homomorphic images of $P_2$, every vertex is a source or a sink, possibly both, 
and we can define a homomorphism to $T_2$ by mapping all sources to the source of $T_2$ and all other vertices to the sink.

A finite relational structure $A$ is called {\em compact}
if the existence of a homomorphism from a structure $B$ to $A$ is equivalent to 
the existence of homomorphisms from all finite substructures of $B$ to $A$.
The axiom of choice implies the all-encompassing compactness theorem, which implies
that every  finite relational structure is compact.
More precisely, the compactness theorem is equivalent to the ultrafilter axiom, 
which states that any filter on a set can be extended to an ultrafilter. It is weaker
than the axiom of choice. In graph theory, the compactness 
of the complete graphs is called the de Bruijn-Erd\H{o}s theorem, after the 1951 result of de Bruijn and Erd\H{o}s~\cite{be}.
Later in 1971, L\"{a}uchli~\cite{lauchli} proved that for any $n \geq 3$, the compactness of the complete graph 
$K_n$ already implies the ultrafilter axiom. Thus, viewed as an axiom, the statement ``$K_3$ is compact'' is stronger than
``$K_2$ is compact'', and it is at least as strong as ``$A$ is compact'' for any finite relational structure $A$.
Recently, the following generalisation of L\"{a}uchli's result was obtained by 
K\'{a}tay, T\'{o}th, and Vidny\'{a}nszky.
\begin{thmC}[\cite{ktv}] \label{ultrafilterthm}
Let $A$ be a finite relational structure. Then the compactness of $A$ implies the ultrafilter axiom
if and only if $A$ does not admit a cyclic polymorphism.
\end{thmC}
\noindent
The ``cyclic polymorphisms'' involved will be defined in Section~\ref{overview}.
They provide the criterion separating the simple from the complex in the dichotomy 
of constraint satisfaction problems. For a given relational structure $A$, the corresponding 
{\em constraint satisfaction problem} $\mbox{CSP}(A)$
is the problem of deciding whether an input structure $B$ admits a homomorphism to $A$.
Bulatov~\cite{bulatov} and Zhuk~\cite{zhuk} independently proved that $\mbox{CSP}(A)$
is polynomial if $A$ admits a cyclic polymorphism, and NP-complete otherwise.
This dichotomy between tractable and intractable constraint satisfaction problems is
the one that had been conjectured by Feder and Vardi~\cite{federvardi}.
Note that in the context of algorithmic complexity, this is a conditional dichotomy,
as it relies on the hypothesis $\mbox{NP}\neq \mbox{P}$ which is not proved.
However, the dichotomy of Theorem~\ref{ultrafilterthm} is unconditional.
Thus the set-theoretic setting supports the hypothesis  $\mbox{NP}\neq \mbox{P}$;
the constraint satisfaction problems that are hypothetically hardest correspond
precisely to the compactness axioms that are provably the most general.

Now at the bottom end of the hierarchy, we find $T_2$, whose compactness does not rely on any axiom
outside of ZF, the axiom system of Zermelo and Fraenkel. The constraint satisfaction problem $\mbox{CSP}(T_2)$
is polynomial, and more precisely, it can be solved by the ``consistency check'' algorithm which will
be presented in Section~\ref{overview}. The consistency check algorithm is arguably the simplest algorithm
for constraint satisfaction, as it is rediscovered by many Sudoku puzzle solvers with no background in computer science.
A relational structure $A$ is said to have ``width $1$'' or ``tree duality'' if $\mbox{CSP}(A)$ can be solved 
by the consistency check algorithm. In~\cite{rtw}, it is shown that if $A$ has width $1$, then 
the statement ``$A$ is compact'' can be proved  in ZF without any
additional axiom. Once again this is a correspondence between algorithmic complexity and
set theory, connecting some of the simplest constraint satisfaction problems with some of the
simplest compactness results. Here we complete this correspondence as follows.
\begin{thm}\label{main}
Let $A$ be a relational structure. If $A$ does not have width $1$, then the axiom ``$A$ is compact''
implies the existence of a set in $\mathbb{R}^3$
that is not Lebesgue measurable.
\end{thm}
\noindent
The non-measurable sets were originally viewed as an undesirable
consequence of the axiom of choice, but their existence already
follows from the compactness of any relational structure that
does not have width $1$. There are models of ZF without the 
axiom of choice, where all sets of $\mathbb{R}^n$ are measurable.
In Section~\ref{measure}, we will provide a
background in measure theory and its use here.
We then give the proof of Theorem~\ref{main} in 
Section~\ref{banach} and conclude with an open problem.

\section{Terminology} \label{overview}
A {\em type} is a finite set $\sigma = \{R_1,\dots,R_m\}$ of 
relation symbols, each with  an {\em arity}  $r(i)$ assigned to it. 
A {\em relational structure of type $\sigma$} is a $(m+1)$-tuple
$A = (V(A);R_1(A),\dots,R_m(A))$ where $V(A)$ 
is a nonempty set called the {\em universe} of $A$, and 
$R_i(A)$ is an $r(i)$-ary relation on $V(A)$ for each $i$.
Let $A$ and $B$ be $\sigma$-structures,
a {\em homomorphism} from $B$ to $A$ is a map 
$\phi: V(B) \rightarrow V(A)$ such that $\phi(R_i(B)) 
\subseteq R_i(A)$ for all $i=1,\dots,m$.
If $V(B)$ is a subset of $V(A)$ and the identity is
a homomorphism from $B$ to $A$,
then $B$ is called a {\em substructure} of $A$.
The constraint satisfaction problem $\mbox{CSP}(A)$
is the problem of determining whether an input structure $B$ admits 
a homomorphism to $A$.

The {\em $n$-th power} $A^n$ of $A$ has $V(A)^n$ as universe, and
for $R_i \in \sigma$, $R_i(A^n)$ consists of all the $r(i)$-tuples $(f_1, \ldots, f_{r(i)})$
such that for all $j$, $(f_1(j), \ldots, f_{r(i)}(j)) \in R_i(A)$. (Note that an element
of $V(A)^n$ can be represented either as a function $f: \{1, \ldots, n\} \rightarrow V(A)$
 or as an array $(a_1, \ldots, a_n)$.)
A {\em polymorphism of arity $n$} is a homomorphism $\phi: A^n \rightarrow A$.
It is {\em cyclic} if it satisfies $\phi(a_1, \ldots, a_n) = \phi(a_2, \ldots, a_n, a_1)$.
The existence of a cyclic polymorphism on $A$ is the criterion
used in Theorem~\ref{ultrafilterthm}. Here we focus on structures that have {\em totally symmetric polymorphisms}
of all arities, that is polymorphisms $\phi: A^n \rightarrow A$ that satisfy
$\phi_n(f) = \phi_n(g)$ whenever $f(\{1, \ldots, n\}) = g(\{1, \ldots, n\})$. 
The existence of totally symmetric polymorphisms of all arities on $A$ is equivalent
to the existence of a single homomorphism $\omega: \mathcal{U}(A) \rightarrow A$, where 
$\mathcal{U}(A)$ is the structure defined as follows. The
universe of $\mathcal{U}(A)$ is the set of nonempty subsets of $V(A)$, and for 
$R_i \in \sigma$, $R_i(\mathcal{U}(A))$ consists of all
the $r(i)$-tuples $(S_1, \ldots, S_{r(i)})$ such that for all $j$ and $a_j \in S_j$, there exists
$(a_1, \ldots, a_{r(i)}) \in R_i(A)$ such that $a_k \in S_k$ for all $k$.
A structure $A$ is said to have {\em width $1$} or {\em tree duality} if there
exists a homomorphism from $\mathcal{U}(A)$ to $A$.

The {\em consistency check} algorithm proceeds as follows to seek a homomorphism from $B$ to $A$.
Define a map $\ell: V(B) \rightarrow \mathcal{P}(V(A))$ listing all possible images of
each element of $V(B)$ under a homomorphism. Each $\ell(b)$ is initialised at $V(A)$, but then these lists 
are recursively trimmed for consistency: If for $R_i \in \sigma$ and $(b_1, \ldots, b_{r(i)}) \in R_i(B)$
there is some $j$ and $a_j \in \ell(b_j)$ for which there is no $(a_1, \ldots, a_{r(i)}) \in R_i(A)$
with $a_k \in \ell(b_k)$ for all $k$, then $a_j$ is removed from $\ell(b_j)$. If at some point in the process
some list $\ell(b)$ becomes empty, it proves that there is no homomorphism from $B$ to $A$. Otherwise,
$\ell$ stabilises at a homomorphism from $B$ to $\mathcal{U}(A)$. If $A$ has width $1$,
$\ell$ can then be composed with $\omega: \mathcal{U}(A) \rightarrow A$ to give a homomorphism
from $B$ to $A$.

The same trick essentialy shows also that $A$ is compact.
Indeed suppose that all finite substructures of $B$ admit homomorphisms to $A$. For $b \in V(B)$,
let $\ell(b)$ be the set of all elements $a$ of $V(A)$ such that for every finite substructure $B'$
of $B$ containing $b$, there exists a homomorphism $\psi: B' \rightarrow A$ such that $\psi(b) = a$.
Since $V(A)$ is finite, $\ell(b)$ must be nonempty, and the map $\ell: B \rightarrow \mathcal{U}(A)$
is easily seen to be a homomorphism. Again composing with $\omega$ gives a homomorphism $\phi: B \rightarrow A$.
This is essentially the proof of compactness of $A$ given in~\cite{rtw}.

To prove Theorem~\ref{main}, we must deduce the existence of the homomorphism
$\omega: \mathcal{U}(A) \rightarrow A$ from the fact that $A$ is compact and that
all sets in $\mathbb{R}^3$ are measurable. We will use adaptations of the
Banach-Tarski paradox and the Steinhaus theorem.

\section{Measurable sets} \label{measure}
We will use the Lebesgue outer measure in $\mathbb{R}^3$. 
The volume of a box $\Pi_{i=1}^3 (a_i,b_i)$ is $\Pi_{i=1}^3 (b_i - a_i)$, 
and the measure $\mu(X)$ of a set $X \subseteq \mathbb{R}^3$ is the infimum of sums of volumes of a 
countable collection of boxes that cover $X$. Its basic properties can be proved in ZF.
\begin{itemize}
\item Monotonicity: if $X \subseteq Y$, then $\mu(X) \leq \mu(Y)$.
\item Extension of box volume: $\mu(\Pi_{i=1}^3 (a_i,b_i)) = \Pi_{i=1}^3 (b_i - a_i)$. 
\item Invariance under translation: $\mu(X + y) = \mu(X)$ for any $y \in \mathbb{R}^3$ 
\item Invariance under rotation: $\mu(MX) = \mu(X)$ for any orthogonal matrix $M$. 
\end{itemize}
A set $X$ is called {\em measurable} if for any set $Y$ we have $\mu(Y) = \mu(Y \cap X) + \mu(Y\setminus X)$.
Restricted to measurable sets, $\mu$ is additive, but here we postulate that all sets are measurable.
\begin{itemize}
\item Postulate of additivity: If $X$ and $Y$ are disjoint, then $\mu(X \cup Y) = \mu(X) + \mu(Y)$.
\end{itemize}
It then follows that $\mu$ is finitely additive. However the countable axiom of choice 
is needed to prove its countable additivity. 

The statement ``all sets in $\mathbb{R}$ are measurable'' is well known to be an axiom in
Solovay's model~\cite{solovay}, along with the principle of dependent choice which is
necessary for the countable additivity of the Lebesgue measure and other results in analysis.
Within Solovay's model, it can be proved that all sets in 
$\mathbb{R}^n$ are also measurable.
However the construction of Solovay's model requires the existence of an inaccessible cardinal as an axiom. 
Here, we do not require the countable additivity of the 
Lebesgue measure and the principle of dependent choice.
According to note 6 of the epilogue of~\cite{moore},
Solovay first found a model of ZF in which all sets in 
$\mathbb{R}$ are measurable,
without the use of an inaccessible cardinal.
This was not published, because Solovay considered
the principle of dependent choice essential.
It is likely that in this earlier model,
all sets in $\mathbb{R}^3$ are also measurable. 
However, based on published results, the
dichotomy established by Theorem~\ref{main} is dependent
on the existence of an inaccessible cardinal.

\section{Banach-Tarski graphs} \label{banach}
The Banach-Tarski paradox is a well-known refutation of the additivity of the $3$-dimensional Lebesgue measure 
in the presence of the axiom of choice. We use its setting constructively in ZF. 
We use orthogonal $3$ by $3$ matrices, more precisely rotation matrices whose action on the sphere $\mathbb{S}_2$ is a rotation
about an axis. There are well-known constructions of pairs
of such rotation matrices which generate a free group.
In turn, a free group on two generators contains a countably infinitely generated free group. 
We will use a group $\mathcal{G}$ of rotation matrices freely generated by a finite set $\{Q_1, Q_2\} \cup \{D_1, \ldots, D_b\}$,
where $b$ will be specified later. We denote $F_{\mathcal{G}}$ the set of fixed points of nonidentity elements of $\mathcal{G}$
on $\mathbb{S}_2$. Then $F_{\mathcal{G}}$ is countable.
For $p$ in $F_{\mathcal{G}}$, the elements of $\mathcal{G}$ that have $p$ as pole commute, and since $\mathcal{G}$ 
is free, this means that they are the powers of a single 
(primitive) element of $\mathcal{G}$. For $p \not \in F_{\mathcal{G}}$, the images of 
$p$ by elements of $\mathcal{G}$ are all distinct. 

Let $\mathcal{N}$ be the normal closure of $\langle Q_1, Q_2 \rangle$ in $\mathcal{G}$. 
Then $\mathcal{G}/\mathcal{N}$  is naturally isomorphic to $\langle D_1, \ldots, D_b \rangle$.
Let $(\mathbb{S}_2\setminus F_{\mathcal{G}})/\mathcal{N}$ be the set of orbits of $(\mathbb{S}_2 \setminus F_{\mathcal{G}})$ under $\mathcal{N}$. 
We will use a density property
of $\mathcal{N}$. Let $\delta(M)$ be the maximum of distances from $p$ to $Mp$ with $p$ in $\mathbb{S}_2$. When $M$ is a rotation matrix,
$\delta(M)$ is realised on the ``equator'' of $\mathbb{S}_2$ when the axis of $M$ is taken as north-south axis.
\begin{lem}\label{tech}
For any rotation matrix $D$ and any $\epsilon > 0$, there is some $M$ in $\mathcal{N}$ such that
$\delta(DM) \leq \epsilon$.
\end{lem}
\begin{proof}
Here for simplicity we work with spherical rather than Euclidian distances; the result is the same.
We first show that the set $F_{\mathcal{N}}$ of fixed points of nonidentity elements of $\mathcal{N}$ is dense in $\mathbb{S}_2$.
Let $p$ be a point of $\mathbb{S}_2$, and suppose that the infimum of spherical distances from $p$ to elements of $F_{\mathcal{N}}$ is
$\iota > 0$. Select $q \in F_{\mathcal{N}}$ at spherical distance 
$\iota + \upsilon < 2\iota$ from $p$. Then the elements of 
$\mathcal{N}$ with $q$ as a pole move the open ball of radius 
$\iota$ around  $p$ surjectively to the annulus 
$a(q,\upsilon,\upsilon+2\iota)$ containing the points at spherical distance strictly between 
$\upsilon$ and $\upsilon + 2\iota$ around $q$. This annulus cannot contain any point of $F_{\mathcal{N}}$
because the orbit of such a point under $\mathcal{N}$ would intersect the open ball of radius $\iota$ around $p$.
If $\upsilon > 0$ and $q'$ is an element of  $F_{\mathcal{N}}$
at distance between $\frac{2}{3} \upsilon$ and $\upsilon$
from $q$, then the elements of $\mathcal{N}$ with $q'$
as a pole move $q$ densely on a circle which intersects
$a(q,\upsilon,\upsilon+2\iota)$. This is a contradiction,
so $\upsilon = 0$ and $a(q,\upsilon,\upsilon+2\iota)$
is the punctured disk $a(q,0,2\iota)$.
Therefore $F_{\mathcal{N}}$ cannot contain an element
$q' \neq \pm q$, because the elements of $\mathcal{N}$
with such a $q'$ as a pole would move $q$ densely on a circle with positive radius, a contradiction.
Therefore $F_{\mathcal{N}} = \{q, -q\}$, but this is impossible
since $\mathcal{N}$ contains the free group 
$\langle Q_1, Q_2 \rangle$.
This shows that $F_{\mathcal{N}}$ is dense in $\mathbb{S}_2$.

Now, let $p$ be a pole of $D$, and for $\epsilon> 0$, let $N$ be an element of $\mathcal{N}$ with a pole $q$ 
at distance less than $\epsilon/3$ from $p$. Let $R$ be the rotation matrix with the great circle through
$p$ and $q$ as equator, that moves $p$ to $q$. Then $R^{-1}NR$ has $p$ as a pole, and angle equal to that of $N$.
Some power $R^{-1}N^kR$ of $R^{-1}NR$ approximates $D^{-1}$ in the sense that the distance between $D^{-1}x$ and
$R^{-1}N^kRx$ is always at most $\epsilon/3$. The distance from $D^{-1} x$ to $N^k x$ is then at most that of
$D^{-1} x$ to $R D^{-1} x$ plus that of $R D^{-1} x$ to $R R^{-1}N^kR x$ plus that of $N^kR x$ to $N^k x$,
with each term at most $\epsilon/3$. Therefore the distance from $x$ to $D N^k x$ is at most $\epsilon$, so that $M = N^k$ satisfies
$\delta(DM) \leq \epsilon.$
\end{proof}

Now we specify $b = \sum_{R_i \in \sigma} (r(i)-1) \cdot \left \lvert R_i(\mathcal{U}(A)) \right \rvert$.
For each  $R_i \in \sigma$ with $r(i) > 1$, for each $(S_1, \ldots, S_{r(i)}) \in R_i(\mathcal{U}(A))$
and for each $j \in \{1, \ldots, r(i)-1\}$ there is a specific generator 
$D\left (R_i,(S_1, \ldots, S_{r(i)}),j\right ) \in \{D_1, \ldots, D_b\}$ representing a connection between
$S_j$ and $S_{j+1}$ in $(S_1, \ldots, S_{r(i)}) \in R_i(\mathcal{U}(A))$.
Let $T$ be the digraph with $(\mathbb{S}_2 \setminus F_{\mathcal{G}})/\mathcal{N}$ as vertex-set, where each vertex
$t$ has $b$ outneighbours $D_1 t, \ldots, D_b t$, and $b$ inneighbours $D_1^{-1} t, \ldots, D_b^{-1} t$.
Then $T$ is an infinite collection of directed regular trees.
We call it a Banach-Tarski graph by virtue of the special structure of its vertex-set.

Next we define a structure $C$ on the universe $V(T) \times V(\mathcal{U}(A))$.
For $R_i \in \sigma$ such that $r(i) = 1$, $R_i(C)$ contains all the elements 
$(t,S)$ such that $t \in V(T)$ and $S \in R_i(\mathcal{U}(A))$. By definition,
$S \in R_i(\mathcal{U}(A))$ means that $s \in R_i(A)$ for all $s \in S$.
For $R_i \in \sigma$ such that $r(i) > 1$,
$R_i(C)$ contains all the elements 
$$\left ((t,S_1),(D_{j_1} t,S_2), (D_{j_2} D_{j_1} t, S_3), \ldots, (D_{j_{r(i)-1}} \cdots D_{j_1} t, S_{r(i)})\right )$$
such that $t \in V(T)$, $(S_1, \ldots, S_{r(i)}) \in R_i(\mathcal{U}(A))$ and
$D_{j_k} = D(R_i,(S_1, \ldots, S_{r(i)}),k)$ for each $k \in \{1, \ldots, r(i)-1\}$.\\

\noindent Locally, the substructure of $C$ induced by $\{t\}] \times V(\mathcal{U}(A))$ only retains
the relations of arity $1$. For any $D \in \{D_1, \ldots, D_b\}$, there is only one connection 
between $\{t\} \times V(\mathcal{U}(A))$ and $\{D t\} \times V(\mathcal{U}(A))$,
specifically between $(t,S_j)$ and $(Dt,S_{j+1})$ where $D = D(R_i,(S_1, \ldots, S_{r(i)}),j)$.
Thus any connected component $C'$ of $C$ is a tree that admits a homomorphism to $A$,
even without the countable axiom of choice. 
Indeed for all $R_i \in \sigma$,
$(S_1, \ldots, S_{r(i)}) \in R_i(\mathcal{U}(A))$, $j \in \{1, \ldots, r(i)\}$
and $s_j \in S_j$, one can select in advance a unique 
$(s_1, \ldots, s_{r(i)}) \in R_i(A)$
such that $s_k \in S_k$ for all $k$. 
Call it $\psi(R_i,(S_1, \ldots, S_{r(i)}), j, s_j)$.
With this, a single initial choice of $(t,S) \in V(C')$
and $\phi(t,S) \in S$ extends canonically to
a homomorphism $\phi: C' \rightarrow A$. 
Of course, infinitely many choices would be needed to get a homomorphism from the whole of $C$ to $A$. However, since the finite substructures of $C$ are contained in finitely many 
tree components, we get the following. 
\begin{lem} \label{compacthom}
If $A$ is compact, then there exists a homomorphism $\phi: C \rightarrow A$.
\hfill $\qed$
\end{lem}
\noindent
For $p \in \mathbb{S}_2\setminus F_\mathcal{G}$, let $f_p: V(\mathcal{U}(A)) \rightarrow V(A)$
be the function defined by $f_p(S) = \phi(p/\mathcal{N},S)$. It preserves the relations of arity $1$,
but otherwise it could be arbitrary. Let $\mathcal{D}$ be the set of matrix products of the form 
$D_{i_1} \cdots D_{i_k}$ such that $0 \leq k \leq b$ and the terms $D_{i_j}$ are 
distinct elements of $\{D_1, \ldots, D_b\}$. For each $f: V(\mathcal{U}(A)) \rightarrow V(A)$
and $D \in \mathcal{D}$, define
$$\fish_{f,D} = \{\lambda p | \lambda \in (0,1], p \in \mathbb{S}_2\setminus F_\mathcal{G} \mbox{ and } 
f_{p} = f \neq f_{D p}\} \subseteq \mathbb{R}^3.$$
There is a finite number of these sets.
Just like in anglers' tales and biblical stories, there is an ambiguity in gauging the measure of a  $\fish_{f,D}$
if it is nontrivial. The following is an adaptation of the theorem of Steinhaus.
\begin{lem}\label{mezero}
If $\fish_{f,D}$ is measurable, then $\mu(\fish_{f,D}) = 0$.
\end{lem}
\begin{proof}
Suppose that $\mbox{Fish}_{f,D}$ is measurable and $\mu(\fish_{f,D}) > 0$. Then there exists a collection 
$\{ B_i \}_{i \in \mathbb{N}}$ of boxes such that $\fish_{f,D} \subseteq \bigcup_{i=0}^{\infty} B_i$ and 
$\sum_{i=0}^{\infty} \mu(B_i) < \frac{11}{10} \mu(\fish_{f,D})$. We can assume
that each box is contained in the ball of radius 2 around the origin.
Let $n$ be an integer such that $\sum_{i=n+1}^{\infty} \mu(B_i) < \frac{1}{10} \mu(\fish_{f,D})$. 
Let $\epsilon$ be the minimum of the side lengths of the boxes $B_1, \ldots, B_n$.
By Lemma~\ref{tech} we can find $M$ in $\mathcal{N}$ such that $\left \lVert p - D M p \right \rVert < \frac{\epsilon}{20}$
whenever $\left\lVert p \right\rVert < 2$. We have $M \fish_{f,D} = \fish_{f,D}$,
and $D  M  \fish_{f,D} \cap \fish_{f,D} = \emptyset$, so that 
$\mu(\fish_{f,D} \cup D  M  \fish_{f,D}) = 2\mu(\fish_{f,D})$. 
However $\fish_{f,D}$ can be split into the two sets 
\begin{eqnarray*}
\head & = & \fish_{f,D} \cap \left ( \bigcup_{i=1}^n B_i \right ), \\
\tail & = & \fish_{f,D} \setminus \head. 
\end{eqnarray*}
For $i = 1, \ldots n$, $B_i \cup D M B_i$ is contained in a box $B_i'$
obtained by adding lengths of $\frac{\epsilon}{20}$ at both ends of each side of $B_i$. Altogether this gives 
$$
\mu(\head \cup D M \head) \leq
\mu\left ( \bigcup_{i=1}^n B_i' \right )
\leq \left ( \frac{11}{10} \right )^3 
\mu\left ( \bigcup_{i=1}^n B_i \right )
\leq \left ( \frac{11}{10} \right )^4 \mu(\fish_{f,D}).
$$
To this we add $\mu(\tail)$ and $\mu(D M \tail)$, both less than 
$\frac{1}{10} \mu(\fish_{f,D})$.
We get  
$$\mu(\fish_{f,D} \cup D M \fish_{f,D}) < 2\mu(\fish_{f,D}),$$ 
a contradiction. 
Therefore $\mu(\fish_{f,D}) = 0$.
\end{proof}

\begin{proof}[Proof of Theorem~\ref{main}] 
Suppose that $A$ is compact and all sets in $\mathbb{R}^3$
are measurable. Then by Lemma~\ref{compacthom},
there exists a homomorphism $\phi: C \rightarrow A$.
This allows to define the functions $f_p$, 
$p \in \mathbb{S}_2\setminus F_\mathcal{G}$
and the sets $\fish_{f,D}$, $f \in V(A)^{V(\mathcal{U}(A))}$, 
$D \in \mathcal{D}$.
By Lemma~\ref{mezero}, these sets all have measure $0$,
and so does their finite union.
The set $\{\lambda p | \lambda \in (0,1], p \in F_{\mathcal{G}}\}$
also has measure $0$ since it can constructively be contained
in a countable collection of boxes of arbitrarily small measure.
Therefore there exists a point $p$ of $\mathbb{S}_2$ 
such that the line $\{\lambda p | \lambda \in (0,1]\}$ 
belongs to neither of these sets.
We then have $f_p = f_{D p}$ for all $D \in \mathcal{D}$. This implies that $f_p$ is a homomorphism
from $\mathcal{U}(A)$ to $A$, since collapsing each set 
$\{ (D p/\mathcal{N},S) | D \in \mathcal{D} \}$ to a single 
element yields a structure isomorphic to $\mathcal{U}(A)$.
\end{proof}

\section{Concluding comments} \label{problems} 
It is nice to see complexity classifications of constraint satisfaction problems align with axioms of set theory, 
with the most general axioms corresponding to the harder problems.
Perhaps this can be extended to promise constraint satisfactions 
as well:
\begin{samepage}
\begin{prob}\label{promise}
For $n \geq 3$, consider the following statement $S(n)$.
\begin{quote}
If all the finite subgraphs of a graph $G$ are $3$-colourable,
then $G$ is $n$-colourable.
\end{quote}
As an axiom, does it imply the ultrafilter axiom?
\end{prob}
\end{samepage}
The first statement $S(3)$ is the usual compactness of $K_3$, 
and L\"{a}uchli's original result~\cite{lauchli} is that
it implies the ultrafilter axiom. Recently,
Bodor~\cite{bodor} proved that $S(4)$ and $S(5)$
also imply the ultrafilter axiom. For the moment,
Problem~\ref{promise} is open for $n \geq 6$.
This aligns with complexity theory. It is known that if 
$P \neq NP$, there is no colouring heuristic that works in polynomial time and is guaranteed to output a proper colouring 
with at most $5$ colours when the input
is $3$-colourable~\cite{bbko}, and it is conjectured that no 
finite number of colours would suffice.
Without non-measurable sets, it is impossible to colour our 
Banach-Tarski graph $T$ with any finite number of colours, 
even though its finite subgraphs are all bipartite.
Therefore, each statement $S(n)$ of Problem~\ref{promise}, 
as an axiom, is outside of ZF.
The complexity conjecture somehow predicts that
it is equivalent to the ultrafilter axiom. 

In descriptive set theory, the concepts of Borel 
relational structures and Borel homomorphisms allow
to work with the axiom of choice rather than
weaker axioms, while still distinguishing the homomorphisms that
can be ``constructed'' from those that exist but cannot be seen.
For instance, the Banach-Tarski graph $T$
can be viewed as structure on $\mathbb{S}_2$ with two relations, adjacency and equivalence.  The set
$\{ \{p, Dp\} : p \in \mathbb{S}_2, D \in \{D_1, \ldots, D_b\}\}$
is closed in $\mathbb{S}_2^2$, so that removing
the countable set 
$\{ \{p, Dp\} : p \in F_{\mathcal{G}}, D \in \{D_1, \ldots, D_b\}\}$
yields a Borel set, corresponding to adjacency in $T$.
Similarly, $\{ \{p, Mp\} : p \in \mathbb{S}_2 \setminus 
F_{\mathcal{G}}, M \in \mathcal{N} \}$ is a countable union of
closed sets minus a countable set, hence a Borel set
corresponding to equivalence in $T$. This makes $T$ a 
Borel relational structure. (The countable set $F_{\mathcal{G}}$
is isolated.)
Assuming compactness, $T$ admits a $2$-colouring,
mapping equivalent points to the same colour and
adjacent points to different colours. However, by
Lemma~\ref{mezero}, for any proper colouring of $T$
with a finite number of colours, one colour class must 
be non-measurable,
hence not a Borel set. Thus $T$ does not admit a 
proper Borel colouring with a finite number of colours.
In a work that parallels Theorem~\ref{main},
Thornton~\cite{thornton2022algebraicapproachborelcsps} 
has shown that the finite relational structures of
width $1$ are precisely those for which
the existence of an arbitrary homomorphism
from a Borel relational structure into them corresponds 
to the existence of a Borel homomorphism.

\bibliographystyle{alphaurl}
\bibliography{w1nms}

\end{document}